\documentclass[12pt]{article}
\usepackage{amssymb}

\newtheorem{thm}{Theorem}

\newtheorem{theorem}{Theorem}[section]
\newtheorem{corollary}[theorem]{Corollary}

\newtheorem{lemma}[theorem]{Lemma}

\newtheorem{example}[theorem]{Example}

\newenvironment{proof}[1][Proof]%
{\par\addvspace{6pt}\noindent{\bf #1.}\hskip\labelsep\ignorespaces}%
{{\hfill $\square$}\par\addvspace{6pt}}

\def\card#1{\vert #1 \vert}
\def\gpindex#1#2{\card {#1\colon #2}}
\def\irr#1{{\rm  Irr}(#1)}

\def\gpcen#1{{\bf Z} (#1)}

\def\I#1#2{{\rm I}_{#1} (#2)}
\def\Ipi#1{\I {\pi} #1}
\def\B#1#2{{\rm B}_{#1} (#2)}
\def\Bpi#1{\B {\pi}#1}

\begin{document}

%Notes- Mark Lewis - April 2004, revised November-December 2008.

\title{Groups where all the irreducible characters are super-monomial}

\author {
       Mark L.\ Lewis
    \\ {\it Department of Mathematical Sciences, Kent State University}
    \\ {\it Kent, Ohio 44242}
    \\ E-mail: lewis@math.kent.edu
       }
\date{December 11, 2008}

\maketitle

\begin{abstract}
Isaacs has defined a character to be super monomial if every
primitive character inducing it is linear.  Isaacs has conjectured
that if $G$ is an $M$-group with odd order, then every irreducible
character is super monomial.  We prove that the conjecture is true
if $G$ is an $M$-group of odd order where every irreducible
character is a $\{ p \}$-lift for some prime $p$.  We say that a
group where irreducible character is super monomial is a super
$M$-group.  We use our results to find an example of a super
$M$-group that has a subgroup that is not a super $M$-group.

MSC primary : 20C15

\end{abstract}

\section{Introduction}

Throughout this note, $G$ will be a finite group.  Following
\cite{text}, we say that a character $\chi$ of $G$ is monomial if it
is induced from a linear character, and a group $G$ is an $M$-group
if every irreducible character of $G$ is monomial.  In
\cite{Issurv}, Isaacs has suggested that if $G$ is an $M$-group of
odd order and $\chi \in \irr G$, then every primitive character
inducing $\chi$ must be linear.

For ease of exposition, we introduce some notation from
\cite{Istanb}: the character $\chi \in \irr G$ is {\it
super-monomial} if every primitive character inducing $\chi$ is
linear.  It is not hard to show that $\chi$ is super-monomial if and
only if every character inducing $\chi$ is monomial.  We say a group
$G$ is a {\it super $M$-group} if every character in $\irr G$ is
super-monomial.  Hence, Isaacs is suggesting that every $M$-group of
odd order is a super $M$-group. As far as we can tell, there has
been little progress on this problem, and the only class of groups
which are known to be super $M$-groups are groups where every
subgroup is an $M$-group.

In this note, we obtain a partial results along the lines suggested
by Professor Isaacs.  In particular, we will find a class of groups
whose members that have odd-order and are $M$-groups will be super
$M$-groups.

To define this class of groups, we need to introduce Isaacs'
$\pi$-theory.  Let $\pi$ be a set of primes, and let $G$ be a
$\pi$-separable group.
%Put references here.
Let $G^\pi$ be the set of elements of $G$ whose orders are
$\pi$-numbers. (Usually, one uses $G^0$ to denote the $\pi$-elements
of $G$, but since we will be restricting both $\pi$ and
$\pi'$-special characters, it is useful to distinguish the
$\pi$-elements and the $\pi'$-elements of $G$.)  Notice that $G^\pi$
is a union of conjugacy classes of $G$. Given a character $\chi$ of
$G$, we define $\chi^\pi$ to be the restriction of $\chi$ to
$G^\pi$. Following Isaacs, we say that these are the $\pi$-partial
characters of $G$.  Isaacs has defined $\Ipi G$ to be the
irreducible $\pi$-partial characters of $G$. These are the
$\pi$-partial characters that cannot be written as the sum of other
$\pi$-partial characters. Isaacs has proved that $\Ipi G$ forms a
basis for the vector space of all complex valued functions on
$G^\pi$. When $\pi = \{p\}^\prime$, the $\Ipi G$ is the set of
irreducible $p$-Brauer characters of $G$; so the $\pi$-partial
characters can be viewed as generalizing Brauer characters to sets
of primes.  Isaacs has shown that induction can be defined for
$\pi$-partial characters, and hence, a $\pi$-partial character can
be said to be primitive if it is not induced from any proper
subgroup.  It can be called monomial if it is induced from a linear
$\pi$-partial character.  It is super-monomial if every
$\pi$-partial character is monomial.

Given a $\pi$-partial character $\phi$, a {\it $\pi$-lift} of $\phi$
is a character $\chi$ so that $\chi^\pi = \phi$.  If $\phi \in \Ipi
G$, then it is easy to see that $\chi \in \irr G$.  On the other
hand, $\chi \in \irr G$ does not imply that $\chi^\pi$ is
irreducible.  Thus, not every irreducible character is a lift of an
irreducible $\pi$-partial character.  On the other hand, Isaacs has
found certain subsets of $\irr G$ that do consist entirely of lifts.
For example, he defines the set $\Bpi G \subseteq \irr G$, and he
proves that the map $\chi \mapsto \chi^\pi$ is a bijection from
$\Bpi G$ to $\Ipi G$.  He has shown that this correspondence
transports information between $\Bpi G$ and $\Ipi G$. In this note,
we will show that when $|G|$ is odd, such information can be
translated to any lift.

In particular, we will prove that if $G$ is a group with odd order
and $\chi \in \irr G$ is a $\pi$-lift, then $\chi$ is
(super-)monomial if and only if $\chi^\pi$ is (super-)monomial. We
will also give an example to show that this is not true if $\card G$
is even.  With the result on lift we will prove the following result
regarding $M$-groups of odd order.

\begin{thm}\label{Theorem C}
Let $G$ be a group of odd order.  Assume that every character in
$\irr G$ is a $\{ p \}$-lift for possibly different primes $p$.  If
$G$ is an $M$-group, then $G$ is a super $M$-group.
\end{thm}

We will use this theorem to produce an example of a super $M$-group
that has a subgroup that is not an $M$-group.

\section{Super $M$-groups and normal subgroups}

Isaacs has conjectured that if $G$ is an $M$-group with $\card G$
odd, then $G$ is a super $M$-group.  If true, this conjecture would
imply that normal subgroups of $M$-groups are $M$-groups.  We will
see that this is a consequence of a more general result.  The
results in this section have been known (at least to Isaacs), but as
far as we know they have not been proved elsewhere.

\begin{lemma} \label{class}
Let ${\cal G}$ be a class of groups that is closed under normal
subgroups, and let ${\cal M}$ be the class of $M$-groups in ${\cal
G}$.  Suppose that every group in ${\cal M}$ is a super $M$-group.
Then ${\cal M}$ is closed under normal subgroups.
\end{lemma}

\begin{proof}
Fix a group $G \in {\cal M}$.  Let $M$ be a normal subgroup of $G$.
If $M = G$, then the result is trivial.  Thus, we may assume that $M
< G$, and we can find a subgroup $N$ so that $M \le N < G$ where $N$
is a maximal normal normal subgroup of $G$.  Since $G$ is an
$M$-group, it is solvable, and so, $\gpindex GN = p$ for some prime
$p$.  As ${\cal G}$ is closed under normal subgroups, $N \in {\cal
G}$.  Consider a character $\theta \in \irr N$.  We know from
\cite{text} that either $\theta$ extends to $G$ or $\theta$ induces
irreducibly to $G$.  If $\theta$ extends to $G$, then there is a
character $\chi \in \irr G$ so that $\chi_N = \theta$.  In
\cite{text}, it shown that if $\chi$ is monomial, then $\theta$ is
monomial.  On the other hand, if $\theta$ induces irreducibly to
$G$, then there is a character $\chi \in \irr G$ so that $\theta^G =
\chi$.  Since $G$ is a super $M$-group, $\chi$ is super-monomial, so
$\theta$ is monomial.  This implies that $N$ is an $M$-group, so $N
\in {\cal M}$.  Since $\card N < \card G$, we can use induction on
$\gpindex GM$ to see that $M \in {\cal M}$.
\end{proof}

Notice that this does not imply that normal subgroups of super
$M$-groups are necessarily $M$-groups.  This will only follow if all
$M$-groups in the class are super $M$-groups.

We can now show that Isaacs' conjecture implies that normal
subgroups of odd order $M$-groups are $M$-groups.

\begin{corollary}
Suppose that every odd order $M$-group is a super $M$-group.  Then
every normal subgroup of an odd order $M$-group is an $M$-group.
\end{corollary}

\begin{proof}
The class of odd order groups is closed under normal subgroups.
Hence, the previous theorem applies to give the result.
\end{proof}

At this time, there have been few examples of super $M$-groups.
Perhaps the easiest way to produce a super $M$-group is the
following.  Since all groups that are supersolvable-by-$A$-groups
are $M$-group, this implies that all of these groups are super
$M$-groups.  (Recall that an $A$-group is a solvable group where all
the Sylow subgroups are abelian.)  In particular, nilpotent groups,
supersolvable groups, and $A$-groups are all super $M$-groups.

\begin{lemma}
Let $G$ is a group where all the primitive characters are linear and
every proper subgroup of $G$ is an $M$-group. Then $G$ is a super
$M$-group.
\end{lemma}

\begin{proof}
Consider a character $\chi \in \irr G$.  If $\chi$ is primitive,
then $\chi$ is linear and hence super monomial. Thus, we may assume
that $\chi$ is not primitive.  Let $H < G$ and $\theta \in \irr H$
so that $\theta^G = \chi$.  Since all proper subgroups of $G$ are
$M$-groups, $H$ is an $M$-group, and $\theta$ is monomial.
Therefore, $\chi$ is monomial.
\end{proof}

\section{Lifts in odd order groups}

We now work to prove Theorem \ref{Theorem C}.  We begin by studying
character in $\B {\pi'}G$ that happen to be $\pi$-lifts.  We prove
that any such character must be real.

\begin{lemma} \label{primelifts}
Let $\pi$ be a set of primes and let $G$ be a $\pi$-solvable group.
Consider the partial character $\phi \in I_\pi (G)$.  Let $\chi \in
B_\pi (G)$ so that $\chi^\pi = \phi$.  If there is a character $\psi
\in B_{\pi'} (G)$ so that $\psi^\pi = \phi$, then $\chi = \bar
{\chi}$.
\end{lemma}

\begin{proof}
Let $n = |G|_\pi$ and $m = |G|_{\pi'}$. Let $\mbox
{\boldmath$Q$}_{\pi}$ be the field obtained by adjoining all complex
$n$th roots of unity to $\mbox {\boldmath$Q$}$, and let $\mbox
{\boldmath$Q$}_{\pi'}$ be the field obtained by adjoining all
complex $m$th roots of unity to $\mbox {\boldmath$Q$}$. Using
Theorem 20.12 of \cite{isaalg}, $\mbox {\boldmath$Q$}_{\pi} \cap
\mbox {\boldmath$Q$}_{\pi'} = \mbox {\boldmath$Q$}$. By Corollary
12.1 of \cite{pisep}, we know that the values of $\chi$ lie in
$\mbox {\boldmath$Q$}_{\pi}$ and the values of $\psi$ lie in $\mbox
{\boldmath$Q$}_{\pi'}$.  Since $\phi = \chi^\pi = \psi^\pi$, we
conclude that the values of $\phi$ lie in $\mbox
{\boldmath$Q$}_{\pi} \cap \mbox {\boldmath$Q$}_{\pi'} = \mbox
{\boldmath$Q$}$.  We now have $\bar\chi^\pi = \bar \phi = \phi$.  We
know that the map from $\Bpi G$ to $\Ipi G$ is a bijection and it is
not difficult to see that $\bar \chi \in \Bpi G$. Therefore,
$\bar\chi = \chi$ as desired.
\end{proof}

We use Lemma \ref{primelifts} to show that if $\card G$ is odd and
$\psi \in \B {\pi'}G$ is a $\pi$-lift, then $\psi$ is linear and a
$\pi$-lift of the principal character.

\begin{corollary} \label{oddprimelifts}
Let $G$ be a group of odd order, and let $\pi$ be a set of primes.
If $\psi \in B_{\pi'} (G)$ and $\psi^\pi \in I_{\pi} (G)$, then
$\psi (1) = 1$ and $\psi^\pi = (1_G)^\pi$.
\end{corollary}

\begin{proof}
Let $\chi \in B_\pi (G)$ so that $\chi^\pi = \psi^\pi$.  By Lemma
\ref{primelifts}, we know that $\chi = \bar \chi$.  It is well-known
that if $|G|$ is odd and $\chi = \bar \chi$, then $\chi = 1_G$ (see
Exercise 3.16 of \cite{text}).  This proves that $\psi^\pi =
(1_G)^\pi$ and thus, $\psi (1) = 1$.
\end{proof}

We now show if $\chi$ is a $\pi$-lift and $\chi$ is primitive, then
$\chi^\pi$ is primitive.

\begin{lemma} \label{primilift}
Let $\pi$ be a set of primes and $G$ a $\pi$-separable group.
Suppose $\chi \in \irr G$ and $\chi^\pi = \phi \in I_\pi (G)$.  If
$\phi$ is primitive, then $\chi$ is primitive.
\end{lemma}

\begin{proof}
Suppose that $H \subseteq G$ and $\theta \in \irr H$ so that
$\theta^G = \chi$. We see that $(\theta^\pi)^G = (\theta^G)^\pi =
\chi^\pi = \phi$.  Since $\phi$ is irreducible, $\theta^\pi$ is
irreducible, and thus, the fact that $\phi$ is primitive implies $H
= G$.  Therefore, $\chi$ is primitive.
\end{proof}

If $\chi \in \irr G$ is primitive, then $\chi$ is $\pi$-factored. If
in addition $\chi \in \Bpi G$, then by Lemma 5.4 of \cite{pisep},
this implies that $\chi$ is $\pi$-special. Using \cite{indct}, it is
not difficult to show that a $\pi$-special character is primitive if
and only if it lifts a primitive $\pi$-partial character. Therefore,
we see that $\chi \in B_\pi (G)$ is primitive if and only if
$\chi^\pi$ is primitive. For groups of odd order, this is true for
any lift.

\begin{lemma} \label{oddprimilift}
Let $G$ be a group of odd order, and $\pi$ a set of primes.  Let
$\chi \in \irr G$ so that $\chi^\pi = \phi \in I_\pi (G)$.  Then
$\chi$ is primitive if and only if $\phi$ is primitive.
\end{lemma}

\begin{proof}
We saw in Lemma \ref{primilift} that if $\phi$ is primitive, then
$\chi$ is primitive.  Conversely, suppose that $\chi$ is primitive.
%Reference here?
We know that $\chi = \theta \psi$ where $\theta$ is $\pi$-special
and $\psi$ is $\pi'$-special and both $\theta$ and $\psi$ is
primitive. Now, $\psi \in B_{\pi'} (G)$.  Also, $\phi = \theta^\pi
\psi^\pi$. Since $\phi$ is irreducible, it follows that $\psi^\pi$
must be irreducible, so we may use Corollary \ref{oddprimelifts} to
see that $\psi (1) = 1$ and $\psi^\pi = (1_G)^\pi$. It follows that
$\phi = \theta^\pi$.  As we mentioned before this lemma,
$\theta^\pi$ is primitive if and only if $\theta$ is primitive.
Because $\theta$ is primitive, we conclude that $\phi = \theta^\pi$
is primitive.
\end{proof}

We can now show that if $\chi \in \irr G$ is a $\pi$-lift and $\chi$
is monomial, then $\chi^\pi$ is monomial.

\begin{lemma} \label{monlift}
Let $\pi$ be a set of primes, and let $G$ be a $\pi$-separable
group. Suppose $\chi \in \irr G$ with $\chi^\pi = \phi \in I_\pi
(G)$.  If $\chi$ is monomial, then $\phi$ is monomial.
\end{lemma}

\begin{proof}
Since $\chi$ is monomial, we can find $H \subseteq G$ and $\lambda
\in \irr H$ so that $\lambda (1) = 1$ and $\lambda^G = \chi$.  We
have that $(\lambda^\pi)^G = (\lambda^G)^\pi = \chi^\pi = \phi$.  It
follows $\phi$ is induced from $\lambda^\pi$ which is linear, so
$\phi$ is monomial.
\end{proof}

We now consider the case where $\card G$ is odd.  We prove that if
$\chi \in \irr G$ is a $\pi$-lift, then $\chi$ is super-monomial if
and only if $\chi^\pi$ is super-monomial.

\begin{lemma} \label{superodd}
Let $G$ be a group of odd order, and let $\pi$ be a set of primes.
Suppose $\chi \in \irr G$ with $\chi^\pi = \phi \in I_\pi (G)$.  If
$\phi$ is super-monomial, then $\chi$ is super-monomial.
\end{lemma}

\begin{proof}
Let $H \subseteq G$ and $\mu \in \irr H$ such that $\mu^G = \chi$
and $\mu$ is primitive.  We see that $(\mu^\pi)^G = \phi$, so
$\mu^\pi \in I_\pi (H)$. By Lemma \ref {oddprimilift}, $\mu^\pi$ is
primitive.  Since $\phi$ is super-monomial, this implies that
$\mu^\pi$ is linear, and so $\mu$ is linear. We conclude that $\chi$
is super-monomial.
\end{proof}

We continue to work in the case where $\card G$ is odd.  We show
that if $\chi \in \irr G$ is a $\{ p \}$-lift where $p$ is prime,
then $\chi$ is monomial if and only if $\chi$ is super-monomial.
This uses the fact that Isaacs proved in \cite{lifts} for $\{ p
\}$-partial characters that monomial is equivalent to
super-monomial.  The only case where the hypothesis that $|G|$ is
odd is used is to show 3 implies 4.

\begin{theorem} \label{mainthm}
Let $p$ be an odd prime, and let $G$ be a group of odd order.
Suppose $\chi \in \irr G$ with $\chi^{\{ p \}} = \phi \in I_{\{ p
\}} (G)$. Then the following are equivalent:

\begin{enumerate}
\item $\chi$ is monomial.
\item $\phi$ is monomial.
\item $\phi$ is super-monomial.
\item $\chi$ is super-monomial.
\end{enumerate}
\end{theorem}

\begin{proof}
Note that (4) implies (1) is obvious, and (1) implies (2) is Lemma
\ref{monlift}.  Isaacs proved (2) implies (3) as Theorem F in
\cite{lifts}, and (3) implies (4) is Lemma \ref{superodd}.
\end{proof}

We now prove a preliminary result for groups where all the $\{ p
\}$-partial characters are monomial, and all the characters in $\irr
G$ are $\{ p \}$-lifts for perhaps different primes $p$.

\begin{theorem}\label{liftsupmon}
Let $G$ be a group with odd order.  Suppose for every prime $p$
dividing the order of $\card G$ that the partial characters in $\I
pG$ is monomial.  Assume that every character in $\irr G$ is a $\{ p
\}$-lift for some prime $p$ that divides $\card G$, not necessarily
the same prime for all characters in $\irr G$.  Then $G$ is super
$M$-group.
\end{theorem}

\begin{proof}
Let $\chi \in \irr G$.  By hypothesis, there is a prime $p$ dividing
$\card G$ so that $\chi^{\{p\}}$ is irreducible.  Also, by
hypothesis, $\chi^{\{p\}}$ is monomial.  We apply Theorem
\ref{mainthm} to see that this implies that $\chi$ is
super-monomial.  This shows that every character in $\irr G$ is
super-monomial, and hence, $G$ is a super $M$-group.
\end{proof}

Next, we obtain Theorem \ref{Theorem C} as a corollary.

\begin{corollary}\label{liftmsup}
Let $G$ be an $M$-group with odd order.  Suppose every character in
$\irr G$ is a $\{ p \}$-lift for some prime dividing $\card G$. Then
$G$ is a super $M$-group.
\end{corollary}

\begin{proof}
In light of Theorem \ref{liftsupmon}, it suffices to show that every
character in $\I pG$ is monomial for all primes $p$ that divide
$\card G$.  However, we know that every character in $\B pG$ is
monomial.  By Theorem \ref{mainthm}, this implies that all partial
characters in $\I {\{p\}}G$ are monomial.
\end{proof}

%
%Question 1:  Let $G$ be a group of odd order, and let $\pi$ be a set
%of primes. Suppose $\chi \in \irr G$ has $\chi^\pi = \phi \in I_\pi
%(G)$.  If $H \subseteq G$ so that $\nu \in I_\pi (H)$ so that $\nu^G
%= \phi$, then does there exist $\theta \in \irr H$ so that
%$\theta^\pi = \nu$ and $\theta^G = \chi$?  (This seems unlikely to
%be true, but it would be interesting to see what a counterexample
%looks like.)

Finally, the following lemma is motivated by a result of Navarro.

\begin{lemma} \label{Navinsp}
Let $\pi$ be a set of primes, let $G$ be a $\pi$-separable group,
and let $H \le G$.  Suppose that $\mu$ is a $\pi$-special character
of $H$, and let $\nu = \mu^\pi$.  Assume that $\phi = \nu^G \in \I
{\pi}G$.  If $\lambda$ is a $\pi'$-special linear character of $H$,
then $(\mu\lambda)^G$ lifts $\phi$.
\end{lemma}

\begin{proof}
We know that $(\mu^G)^\pi = (\mu^\pi)^G = \nu^G = \phi$.  Hence, it
follows that $\mu^G$ must be irreducible.   Since $\lambda$ is
linear and $\pi'$-special, it follows that $\lambda^\pi =
(1_G)^\pi$, and so, $(\mu\lambda)^\pi = \mu^\pi \lambda^\pi = \nu
(1_G)^\pi = \nu$.  We see that $((\mu\lambda)^G)^\pi = (\mu^G)^\pi =
\phi$, and so $(\mu\lambda)^G$ lifts $\phi$.
\end{proof}

We know that every character is induced from a primitive character.
Suppose $\chi \in \irr G$ and $\theta \in \irr H$ is primitive so
that $\theta^G = \chi$.  If $\chi$ is a $\pi$-lift for some set of
primes $\pi$, then $\theta$ is a $\pi$-lift.  When $|G|$ is odd,
Lemma \ref{oddprimilift} shows that $\theta^\pi$ must be primitive.
We know that $\theta$ is $\pi$-factored.  Hence, $\theta =
\theta_\pi \theta_{\pi'}$ where $\theta_\pi$ is $\pi$-special and
$\theta_{\pi'}$ is $\pi'$-special.  By Lemma \ref{oddprimelifts},
$\theta_{\pi'}$ is linear.  Hence, $\chi$ has the form of the lift
found in Lemma \ref{Navinsp}.  In particular, if $|G|$ is odd and
$\phi \in \Ipi G$, then all lifts of $\phi$ arise in the fashion of
Lemma \ref{Navinsp}.  We will see that when $|G|$ is not odd, they
do not necessarily arise in this fashion.

\section{Examples}

We now find groups where all the irreducible characters are $\{ p
\}$-lifts for various primes $p$.  We begin with a technical lemma
that will be various useful.

\begin{lemma}\label{plft}
Let $G$ be a $p$-solvable group for some prime $p$.  Suppose that
$G$ has a normal subgroup $N$ so that $G/N$ is abelian. Let $\theta
\in \B {\{p\}}N$ and let $T$ be the stabilizer of $\theta$ in $G$.
Assume that every character in $\irr {T \mid \theta}$ is an
extension of $\theta$. Then every character in $\irr {G \mid
\theta}$ is a $\{ p \}$-lift.
\end{lemma}

\begin{proof}
Consider a character $\chi \in \irr {G \mid \theta}$.  We use
Corollary 6.4 of \cite{pisep}, to see that $\irr {G \mid \theta}$
contains a character $\psi \in \B {\{ p \}}G$. Let $\gamma, \delta
\in \irr {T \mid \theta}$ so that $\gamma^G = \psi$ and $\delta^G =
\chi$. By Gallagher's theorem, there exists $\lambda \in \irr {T/N}$
so that $\delta = \gamma \lambda$.  Since $G/N$ is abelian,
$\lambda$ extends to $\epsilon \in \irr {G/N}$.  By Problem 5.3 of
\cite{text}, we have $\chi = \delta^G = (\gamma \lambda)^G =
\gamma^G \epsilon = \psi \epsilon$. We conclude that $\chi^{\{p\}} =
\psi^{\{p\}} \epsilon^{\{p\}}$. Since $\psi^{\{p\}}$ is irreducible
and $\epsilon$ is linear, it follows that $\chi^{\{p\}}$ is an
irreducible $\{ p \}$-partial character.
\end{proof}

In this first example, all the irreducible characters are $\{ p
\}$-lifts for a single prime $p$.  Notice that this includes all
Frobenius groups whose Frobenius kernels are $p$-groups and whose
Frobenius complements are abelian.

\begin{lemma}\label{plifts}
Let $G$ be a group.  Suppose that $G$ has a normal subgroup $N$ so
that $N$ is a $p$-group for some prime $p$ and $G/N$ is an abelian
$p'$-group. Then every character in $\irr G$ is a $\{ p \}$-lift.
\end{lemma}

\begin{proof}
Consider a character $\chi \in \irr G$.  Let $\theta$ be an
irreducible constituent of $\chi_N$.  By Corollary 8.16 of
\cite{text}, we know that $\theta$ has an extension in $\irr T$, and
by Gallagher's theorem (Corollary 6.17 of \cite{text}), every
character in $\irr {T \mid \theta}$ is an extension of $\theta$. By
Lemma \ref{plft}, we conclude that $\chi$ is a $\{ p \}$-lift as
desired.
\end{proof}

We now have an example of a super $M$-group that has a subgroup that
is not an $M$-group.

\begin{example} \rm
Let $p$ and $q$ be odd primes, and suppose that $q$ divides $p+1$.
Take $E$ to be an extra-special group of order $p^3$ of exponent
$p$. Let $E_1$ and $E_2$ be copies of $E$, and take $D$ to be the
central product of $E_1$ and $E_2$ so that $D$ is extra-special of
order $p^5$ and exponent $p$. Now, $E$ has an automorphism $\sigma$
of order $q$ that centralizes $\gpcen E$, and there is an
automorphism $s$ on $D$ so that $s$ acts like $\sigma$ on $E_1$ and
$\sigma^{-1}$ on $E_2$.  Let $G$ be the semi-direct product of
$\langle s \rangle$ acting on $D$.

For $i = 1,2$, let $K_i$ be the subgroup generated by $E_i$ and $s$.
We claim that $K_i$ is not an $M$-group. To see that $K_i$ is not an
$M$-group, observe that $E_i$ has irreducible characters of degree
$p$ that extend to $K_i$, and that $K_i$ does not have any subgroups
of index $p$.  Next, we claim that $G$ is an $M$-group.  It is not
difficult to see that the degrees of the irreducible characters of
$G$ are $1$, $q$, and $p^2$ and that the characters of degree $1$
and $q$ are monomial.   To see that the characters of degree $p^2$
are monomial, we note that $s$ normalizes an abelian subgroup $A$ of
index $p^2$ in $E$, and it is not difficult to see that the
character is induced from a linear character of $\langle s \rangle
A$.

By Lemma \ref{plifts}, every character in $\irr G$ is a $\{ p
\}$-lift. By Corollary \ref{liftmsup}, $G$ is a super $M$-group.
Notice that we now have an example of a super $M$-group that has a
subgroup that is not an $M$-group.
\end{example}

We present an example of group with Fitting height $3$ so that all
the irreducible characters are $\{p\}$-lifts for (perhaps different)
primes $p$.

\begin{lemma} \label{fitthree}
Let $G$ be a group.  Suppose that $G$ has normal subgroups $N \le M$
so that $N$ is a $p$-group, $M$ is a Frobenius group whose Frobenius
kernel in $N$, $N/M$ is a $q$-group for some prime $q$, and $G/M$ is
a cyclic $q'$-group.  Then every character in $\irr G$ is either a
$\{ p \}$-lift or a $\{ q \}$-lift.
\end{lemma}

\begin{proof}
Notice that Lemma \ref{plifts} can be applied to $G/N$ to see that
every irreducible character in $\irr {G/N}$ is a $\{ q \}$-lift.
Assume that $\chi \in \irr G$ and that $N$ is not contained in the
kernel of $\chi$.  Let $\nu$ be an irreducible constituent of
$\chi_N$.  Since $M$ is a Frobenius group, we have $\nu^M \in \irr
M$, and by Corollary 6.4 of \cite{pisep}, we conclude that $\nu^M
\in \B {\{p\}}M$.  Because $T/M$ is cyclic, we know from Corollary
11.22 of \cite{text} that all of the characters in $\irr {T \mid
\nu}$ are extensions of $\nu^M$.  By Lemma \ref{plft}, it follows
that every character in $\irr {G \mid \nu}$, in particular $\chi$,
is a $\{p\}$-lift.
\end{proof}

We next present an example that shows that the hypothesis that $|G|$
is odd is necessary in in Lemma \ref{oddprimelifts}, Lemma
\ref{oddprimilift}, and Theorem \ref{superodd}.

\begin{example} \rm
Take $G$ to be ${\rm GL}_2 (3)$, and write $S$ and $Q$ for the
normal subgroups of $G$ corresponding to ${\rm SL}_2 (3)$ and the
quaternions. Let $\chi$ be the unique character in $\irr {G/Q}$ that
has degree $2$. It is not difficult to see that $\chi$ is induced by
a linear character $\lambda \in \irr {S/Q}$, and that $\lambda$ is
$3$-special.  In \cite{pisep}, Isaacs showed that $\irr {G \mid
\lambda}$ has a unique character in $\B {\{3\}}G$, and since $\chi$
is the unique irreducible constituent of $\lambda^G$, it follows
that $\chi \in B {\{3\}}G$.  We write $\phi = \chi^{\{3\}} \in \I
{\{3\}}G$.

Let $\theta$ be the unique character in $\irr Q$ having degree $2$.
We know that $\theta$ is invariant in $G$, so $\theta$ extends to
$S$. Obviously, $\theta$ is $\{2\}$-special (i.e., $\theta$ is
$\{3\}'$-special). Using \cite{pisep}, $\theta$ has a unique
$\{2\}$-special extension $\hat\theta \in \irr S$.  Furthermore, by
its uniqueness, $\hat\theta$ will be $G$-invariant, and so it has an
extension $\psi \in \irr G$.  (In fact, it has two extensions, and
either will work for our purposes.)  Applying \cite{pisep}, we see
that $\psi$ is $\{2\}$-special, and so, $\psi \in B_{\{3\}'} (G)$.
Notice that $\psi$ is now an extension of $\theta$, so $\psi (1) =
2$.  Now, $G$ has two nontrivial conjugacy classes of $3$-elements,
and so, it is not difficult to see that $\psi^{\{3\}} = \phi$.  This
shows that Corollary \ref{oddprimelifts} is false without the
oddness assumption.  It is not difficult to see that $\psi$ is in
fact primitive.  Since $\phi$ is not primitive, it follows that
Lemma \ref{oddprimilift} is false when $|G|$ is not odd.  Finally,
$\phi$ is monomial, and so, super-monomial, so this shows that the
oddness is needed in Lemma \ref{superodd} and in showing 3 implies 4
in \ref{mainthm}.
\end{example}

\section{Questions}

\begin{enumerate}
\item Is the converse of Lemma \ref{class} true?  In particular, if
${\cal M}$ is closed under normal subgroups, then is it true that
all the groups in ${\cal M}$ are super $M$-groups?

\item If $G$ a super $M$-group and $N$ is a normal Hall subgroup of
$G$, then is $N$ a super $M$-group?

\item Assume the situation of Lemma \ref{Navinsp}.  If $\lambda$ and
$\delta$ are $\pi'$-special characters of $H$, under what hypotheses
will $(\mu\lambda)^G = (\mu\delta)^G$ imply that $\lambda = \delta$?
Navarro has shown this to be the case when $H$ is the nucleus of the
$B_\pi (G)$-character of $G$ that lifts $\phi$.

\item What can be said regarding groups where every irreducible
character is a $\{p\}$-lift for (different) primes $p$?  Lemma
\ref{plifts} shows that there is no bound on the derived length of
such groups.  Lemma \ref{fitthree} gives examples with Fitting
height $3$.  Do there exist examples with larger Fitting heights? Do
there exist examples of such groups with Fitting height $3$ that are
$M$-groups and have subgroups that are not $M$-groups?
\end{enumerate}

\end{document}